\documentclass[12pt]{amsart}

\usepackage{enumerate,url,amssymb,  nicefrac, mathrsfs, graphicx, pdfsync, color}
\newtheorem{theorem}{Theorem}[section]
\newtheorem{lemma}[theorem]{Lemma}
\newtheorem*{lemma*}{Lemma}

\newtheorem{proposition}[theorem]{Proposition}
\newtheorem{corollary}[theorem]{Corollary}
\newtheorem{thm}{Theorem}[section]

\numberwithin{equation}{section}

\theoremstyle{definition}
\newtheorem{definition}[theorem]{Definition}
\newtheorem{example}[theorem]{Example}

\theoremstyle{remark}
\newtheorem{remark}[theorem]{Remark}

\numberwithin{equation}{section}

\newcommand{\abs}[1]{\lvert#1\rvert}

\newcommand{\C}{\mathbb{C}}
\newcommand{\D}{\partial}

\newcommand{\W}{\mathscr{W}}

\newcommand{\R}{\mathbb{R}}
\newcommand{\X}{\mathbb{X}}
\newcommand{\U}{\mathbb{U}}

\newcommand{\BB}{\mathbb{B}}
\newcommand{\Y}{\mathbb{Y}}

\newcommand{\dtext}{\textnormal d}

\newcommand{\onto}{\xrightarrow[]{{}_{\!\!\textnormal{onto\,\,}\!\!}}}

\newcommand{\bydef}{\stackrel {\textnormal{def}}{=\!\!=} }
\DeclareMathOperator{\diam}{diam}

\DeclareMathOperator{\loc}{loc}

\DeclareMathOperator{\osc}{osc}

\def\le{\leqslant}
\def\ge{\geqslant}
\def\rr{{\mathbb R}}
\def\rn{{{\rr}^n}}

\def\cl{{\mathscr L}}

\def\fz{\infty}

\def\loc{{\mathop\mathrm{\,loc\,}}}

\def\bint{{\ifinner\rlap{\bf\kern.25em--}
\int\else\rlap{\bf\kern.45em--}\int\fi}\ignorespaces}

\def\bbint{{\ifinner\rlap{\bf\kern.25em--}
\hspace{0.078cm}\int\else\rlap{\bf\kern.45em--}\int\fi}\ignorespaces}

\def\diam{{\mathop\mathrm{\,diam\,}}}

\def\r{\right}
\def\lf{\left}

\begin{document}

\title[Deformations of Bi-conformal Energy and Cusps]{Creating and Flattening \\ Cusp Singularities by \\Deformations of Bi-conformal Energy}

\author[T. Iwaniec]{ Tadeusz Iwaniec}
\address{Department of Mathematics, Syracuse University, Syracuse,
NY 13244, USA }
\email{tiwaniec@syr.edu}

\author[J. Onninen]{Jani Onninen}
\address{Department of Mathematics, Syracuse University, Syracuse,
NY 13244, USA and Department of Mathematics and Statistics, P.O.Box 35 (MaD) FI-40014 University of Jyv\"askyl\"a, Finland}
\email{jkonnine@syr.edu}

\author[Z. Zhu]{Zheng Zhu}
\address{Department of Mathematics and Statistics, P.O.Box 35 (MaD) FI-40014 University of Jyv\"askyl\"a, Finland}
\email{zheng.z.zhu@jyu.fi}

\thanks{T. Iwaniec was supported by the NSF grant DMS-1802107.
J. Onninen was supported by the NSF grant DMS-1700274.}




\maketitle

\begin{abstract}
 Mappings of bi-conformal energy form the  widest class of homeomorphisms that one can hope to  build a viable extension of Geometric Function Theory  with connections to mathematical models of Nonlinear Elasticity. Such mappings  are  exactly  the ones with finite conformal energy and integrable inner distortion. It is  in this way, that our studies  extend the applications of  quasiconformal homeomorphisms to the degenerate elliptic systems of PDEs. The present paper searches a  bi-conformal variant of the Riemann Mapping Theorem,  focusing on domains with exemplary  singular boundaries that are not quasiballs. We establish the sharp description of boundary  singularities that can be  created and flattened by mappings of bi-conformal energy.
\end{abstract}


\section{Introduction}
We are concerned with orientation preserving homeomorphisms $\,h : \mathbb X \onto \mathbb Y\,$ between bounded domains $\,\mathbb X , \mathbb Y \subset \mathbb R^n \,, n \geqslant 2\,$,  of Sobolev class $\,\mathscr W^{1,p} (\mathbb X, \mathbb Y)\,, 1 \leqslant p \leqslant \infty\,$.
\subsection{Quasiconformal Deformations} Of particular interest are homeomorphisms of finite $\,n\,$-harmonic energy; that is, with $\,p = n\,$.
\begin{equation}\label{conformalenergy}
\mathsf E_{\X}[h] \bydef  \int_\X \abs{Dh(x)}^n \, \dtext x  \,< \infty\, .
\end{equation}
Hereafter the symbol  $\,| Dh(x) |\,$ stands for the operator norm of the differential matrix $\,Dh(x) \in \mathbb R^{n\times n}\,$  called the \textit{deformation gradient}. This  integral is invariant under the conformal change of variables in the \textit{reference configuration} $\X$ (not in the \textit{deformed configuration} $\Y$). That is, $\mathsf E_{\X'}[h']=\mathsf E_{\X}[h] $, where $\,h'=h \circ \varphi $ for a conformal transformation $\varphi \colon \X' \onto \X$. This motivates our calling $\mathsf E_{\X} [h]\,$  \textbf{conformal energy} of $\,h\,$. Mappings of conformal energy arise naturally in Geometric Function Theory (GFT) for many reasons~\cite{AIMb, GeVa, HKb, IMb, Reb}.
\begin{definition}
A \textit{Sobolev homeomorphism}  $\,h \colon \X \onto \Y\,$; that is, of class $\,\mathscr W^{1,1}_{\textnormal{loc}}(\mathbb X, \mathbb Y)\,$,  is said to be   \textit{quasiconformal}  if there exists a constant $1 \le \mathcal K < \infty$ so that for almost every  $\,x \in \X\,$ it holds:
\[\abs{Dh(x)}^n \le \mathcal K\, J_h(x) \;,\qquad \textnormal{where} \,\;\;\;J_h(x) = \textnormal{det}\, Dh(x)   . \]
\end{definition}
 Recall that the Jacobian determinant $\,J_h(x)\,$ of any Sobolev homeomorphism  $\,h\in \W_{\loc}^{1,1} (\X, \R^n)\,$ is locally integrable. Actually,
if the deformed configuration $\Y=h(\X)$ has finite volume the Jacobian is globally integrable and
\begin{equation}
\int_\X J_h(x) \, \dtext x \le \abs{\Y}.\nonumber
\end{equation}
In particular, every quasiconformal map $\,h \colon \X \onto \Y\,$ has finite conformal energy:
\begin{equation}
\mathsf E_{\X}[h] =  \int_\X \abs{Dh(x)}^n \, \dtext x  \, \leqslant \mathcal K \int _\X J_h(x)  \dtext x\, = \;\mathcal K\,\abs{\Y}\,
\end{equation}
\subsection{Mappings of Bi-conformal Energy}

The remarkable feature of a quasiconformal mapping is that its inverse $\,f \bydef h^{-1} \,: \mathbb Y \onto \mathbb X\,$ is
also quasiconformal. In particular, both $\,h\,$ and $\,f\,$  have finite conformal energy. Their sum
\begin{equation}\label{eq:bi-conformal}
\mathsf E_{\X \Y }[h] \bydef  \int_\X \abs{Dh(x)}^n \, \dtext x + \int_\Y \abs{Df(y)}^n \, \dtext y \, \bydef \mathsf E_{\Y \X }[f] .
\end{equation}
will be called \textit{bi-conformal energy} of $\,h\,$.

This leads us to a viable extension of GFT with connections to mathematical models of Nonlinear Elasticity (NE)~\cite{Anb, Bac, Cib, MHb}.
\begin{definition} A   homeomorphism $h \colon \X \onto \Y$ in $\W^{1,n} (\X, \R^n)$, whose inverse $f=h^{-1} \colon \Y \onto \X$ also belongs to $\W^{1,n} (\Y, \R^n)$ is called a \textit{mapping of  bi-conformal energy}.
\end{definition}
 It is equivalent to saying that the inner distortion function of $\,h\,$  is integrable over $\,\mathbb X\,$ and the inner distortion function of $\,f\,$ is integrable over $\,\mathbb Y\,$. For a precise statement (Theorem \ref{thm1.2} below) we need some definitions:
 \subsection{Inner Distortion}
Consider a Sobolev mapping $\,h \in \mathscr W^{1,1}_{\textnormal{loc}} (\mathbb X, \mathbb R^n)\,$ and its \textit{co-differential} $\,D^{\sharp} h(x)  \in \mathbb R^{n\times n}\,$- the matrix determined by Cramer's rule $ D^\sharp h\circ
 Dh =  J_h(x)\,\bf I $.
 \begin{definition}
  The inner distortion of $\,h\,$ is the smallest measurable function $\,K_{_I} (x) =  K_{_I}(x,h)  \in [1, \infty]\,$ such that
   \begin{equation}\label{eq:innerdistortion}
 \abs{D^\sharp h(x)}^n \le K_{_I}(x) J_h(x)^{n-1} \;,\qquad \textnormal{for almost every } x \in \X \, .
 \end{equation}
\end{definition}

The question of finite inner distortion merely asks for the co-differential  $\,D^{\sharp} h(x) = 0\,$ at the points where the Jacobian $J_h(x)=0$. However, for $n \ge 3$, the differential $Dh(x)$ need not vanish if $D^\sharp h (x)=0$.  \\

A formal algebraic computation reveals that the pullback of the $n$-form $K_{_I}(x,h) \, \dtext x \in \wedge ^n \X$ via the inverse mapping $f \colon \Y \onto \X$ equals $\abs{Df(y)}^n \, \dtext y \in \wedge ^n\Y$.
 This observation  is the key to the  fundamental equality between the $\,\mathscr L^1\,$-norm of $K_{_I}(x,h)\,$ and conformal energy of the inverse map  $\,f\,$, which is usually   derived under various regularity assumptions ~\cite{AIMO,  CHM, HK, HKO, Onreg}. We shall state and prove it in the following form:
 \begin{theorem}\label{thm1.2}
 Let $h \colon \X \to \Y$ be an orientation-preserving  homeomorphism in the Sobolev space $\W^{1,n} (\X, \R^n)$, $n \ge 2$. Then the inner distortion of $h$ is integrable if and only if the inverse mapping $f=h^{-1} \colon \Y \to \X$ has finite conformal energy. Furthermore, we have
 \begin{equation}\label{eq:identity}\int_\Y \abs{Df(y)}^n \, \dtext y = \int_{\X} K_{_I}(x,h) \, \dtext x  \, .   \end{equation}
 \end{theorem}
 The interested reader is referred to \cite{IS} for planar mappings with integrable distortion (Stoilow factorization). The following corollary is immediate.
 
\begin{corollary}\label{quasiconformal}
A homeomorphism $h \colon \X \onto \Y$ of  class $\W^{1,n} (\X, \R^n)$ is quasiconformal if and only if  with $K_{_I}(\cdot, h) \in \mathscr L^\infty (\X)$.
\end{corollary} 
 \subsection{Hooke's Low for Materials of Conformal Stored-Energy} In a different direction, the principle of hyper-elasticity is to minimize the given stored energy functional  subject to deformations $\,h : \mathbb X \onto \mathbb Y\,$ of domains made of elastic materials,  \cite{Anb, Bac, Cib, MHb}.  Here we  take on stage the materials of \textit{conformal stored-energy}. This means that the bodies can endure only deformations $\,h : \mathbb X \onto \mathbb Y\,$  whose  gradient $\,Dh\,$ is integrable with power $\,n\,$ (the dimension of the deformed body). A deformation of infinite $n$-harmonic energy would break the internal structure of the material causing permanent damage.  There are examples galore in which one can return the deformed body to its original shape by a deformation of finite conformal energy, but not necessarily via the inverse mapping $\,f \bydef  h^{-1}  : \mathbb Y \onto \mathbb X\,$. The inverse map need not even belong to $\,\mathscr W^{1,n}(\mathbb Y, \mathbb R^n) \,$. On the other hand the essence of Hooke's Low is reversibility. Accordingly, we wish that both $\,h\,$ and $\,f = h^{-1}\,$ have finite conformal energy. Call this model \textit{$\,n\,$-harmonic hyper-elasticity}. It is from this point of view that we arrive a the following $\,n\,$-dimensional variant of the conformal Riemann mapping problem.

 \subsection{Mapping Problems}
 Let $\,\mathbb X , \mathbb Y \subset \mathbb R^n\,$ be bounded domains of the same topological type.  For each of the three problems below find conditions on the pair ($\,\mathbb X, \mathbb Y) \,$ to ensure that:
 \begin{itemize}
 \item [P1)] There exists a bi-Lipschitz deformation $\,h :  \mathbb X \onto \mathbb Y\,$
 \item [P2)] There exists a quasiconformal deformation $\,h :  \mathbb X \onto \mathbb Y\,$
 \item [P3)] There exists  deformation $\,h :  \mathbb X \onto \mathbb Y\,$ of bi-conformal energy
 \end{itemize}
 The following inclusions  P1) $\Longrightarrow$ P2) $\Longrightarrow P3)$ are straightforward.
 \subsection{Ball with Inward Cusp} We shall distinguish a horizontal coordinate axis in $\R^n$,
\begin{equation}
\R^n= \R \times \R^{n-1} =\{(t,x) \colon t \in \mathbb R \textnormal{ and }  x=(x_1, \dots, x_{n-1} ) \in \R^{n-1}   \}\nonumber
\end{equation}

and introduce the
notation
\begin{equation}
 \rho = \abs{x} \bydef\sqrt{ x_1^2 + x_2^2 + \cdots + x^2_{n-1}} .\nonumber
\end{equation}
Consider  a strictly increasing function  $u \colon [0, \infty) \onto [0,\infty)$ of class $\mathscr C^1 (0, \infty) \cap \mathscr C [0, \infty) $. We assume that   $u'$ is increasing in $(0,\infty)$  and
\[ \lim_{\rho \searrow 0} u'(\rho)  = 0 \, . \]
To every such function there corresponds an $(n-1)$-dimensional surface of revolution ${\bf S}_u \in \R_+ \times \R^{n-1}$
\[{\bf S}_u \bydef \{ (t,x) \in \R_+ \times \R^{n-1}  \colon  \,\abs{x} = \rho\;, \;\rho = u(t) \,  \; \}  \,,\; \textnormal {where}\; \mathbb R_+ =[0,\fz) . \]

We shall refer to  ${\bf S}_u$ as {\it a model  cusp} at the origin. Let us emphasize that the case $\, \limsup_{\rho \searrow 0} u'(\rho)  > 0 \,\,$ is \textbf{excluded} from this definition. We may (an do) rescale $u$ so that $u(1)=1$. The model inward cuspy ball is defined by
\[\mathbb B_u \bydef \mathbb B \setminus \{(t,x) \in \R_+ \times \R^{n-1}  \colon \; \abs{x} = \rho \;,\;\rho \le u(t) \, \} \, . \]

 \subsection{Bi-Lipschitz Deformations}  There is no bi-Lipschitz transformation of a cuspy ball (inward or outward as in Figure \ref{fig:cups})  onto a ball without cusp. We say that a cusp cannot be flatten via bi-Lipschitz deformation.
 \begin{figure}[h]
\centering
\includegraphics[width=10cm]{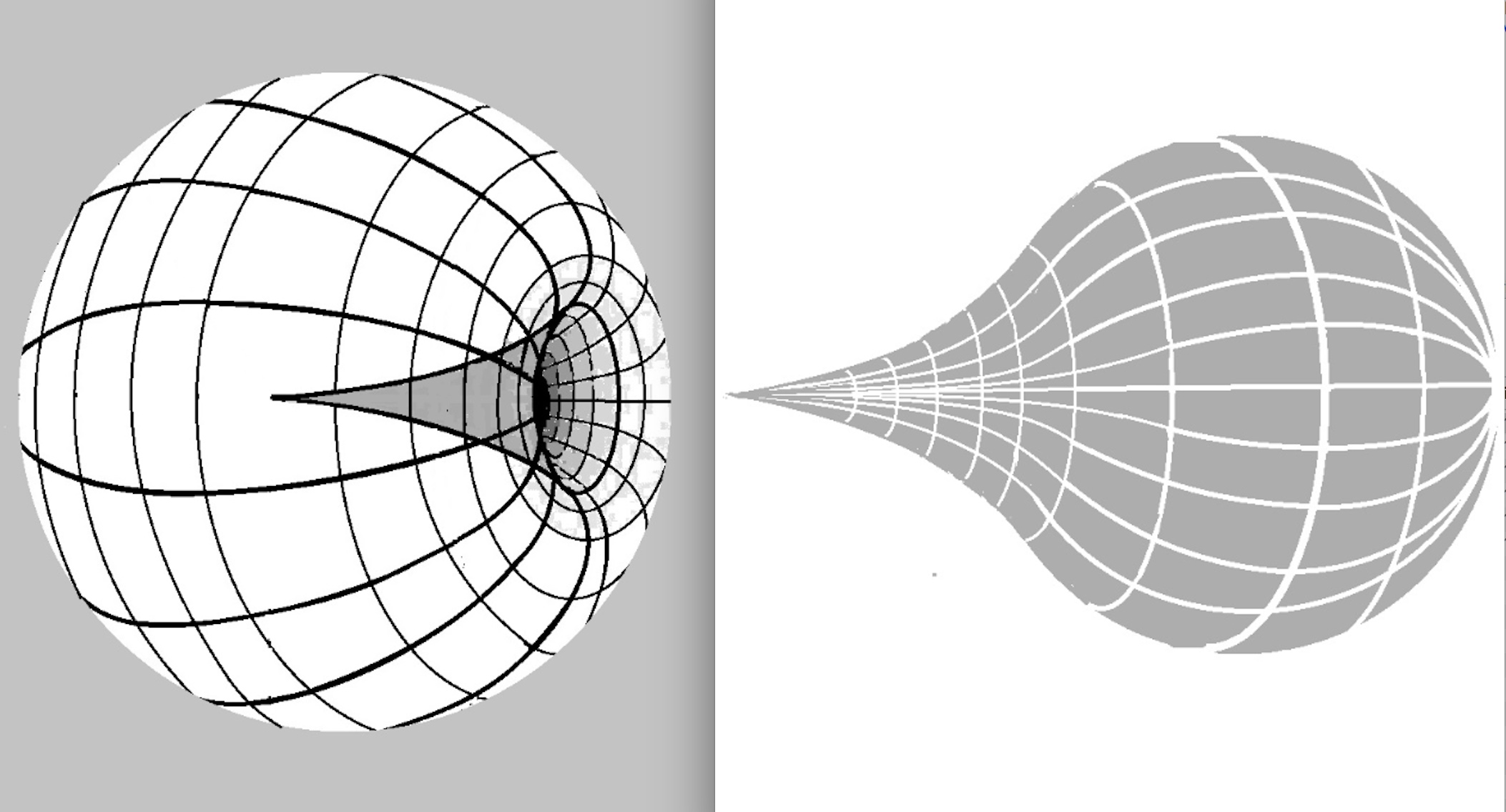}
\caption{Inward and outward cusp  in a ball.}
\label{fig:cups}
\end{figure}
However, there always exists a Lipschitz homeomorphism of a cuspy ball onto a round ball and there is a Lipschitz homeomorphism of the round ball onto the cuspy ball; but these two deformations cannot be inverse to each other.  The same pertains to a \textit{degenerate cusp} defined by $\,u \equiv 0\,$, as in  Figure  \ref{fig:slit}. In this degenerate case, if there would exist a bi-Lipschitz mapping  $\,h: \mathbb B\,\onto \,\mathbb B \setminus \mathbf I\,$,  it would extend as a homeomorphism of $\,\partial \mathbb B\,$ onto $\,\partial (\mathbb B \setminus \mathbf I)\,$, $n \ge 3$, see~\cite{IOb} for more details. It is clear that the conflicting topology of the boundaries is an obstruction to the existence of a bi-Lipschitz deformation. This fact, is also valid for deformations of bi-conformal energy, but it  requires additional arguments, see Theorem~\ref{thm:slitbiconf}.

\subsection{Quasiballs}
 There is a broad literature dealing with $n$-dimensional quasiconformal
variants of the Riemann Mapping Theorem.  F. W. Gehring
and J. V\"ais\"al\"a~\cite{GeVa} raised the question: {\it Which domains $D \subset \R^n$ are quasiconformally equivalent with the unit ball $\mathbb B \subset \R^n$? } Such domains $D$ are called {quasiballs}. The interested reader is referred to the recent book by F. W. Gehring, G.
Martin and B. Palka~\cite{GMPb}. The  Riemann Mapping Theorem gives a complete answer to this question when $n=2$. If $D\varsubsetneq \C$ is a simply connected  domain, then there exists a conformal mapping $h \colon \mathbb B \onto D$.   It is, however, a highly nontrivial question when a domain  $D \subset \R^n$ is a quasiball when  $n \ge 3$.  Among geometric obstructions
are the inward cusps. Indeed, F. W. Gehring
and J. V\"ais\"al\"a~\cite{GeVa} proved that  a ball with inward cusp is not a quasiball. A ball
with outward cusp, however, is  always a quasiball.\\

\begin{figure}[h]
\centering
\includegraphics[width=13cm]{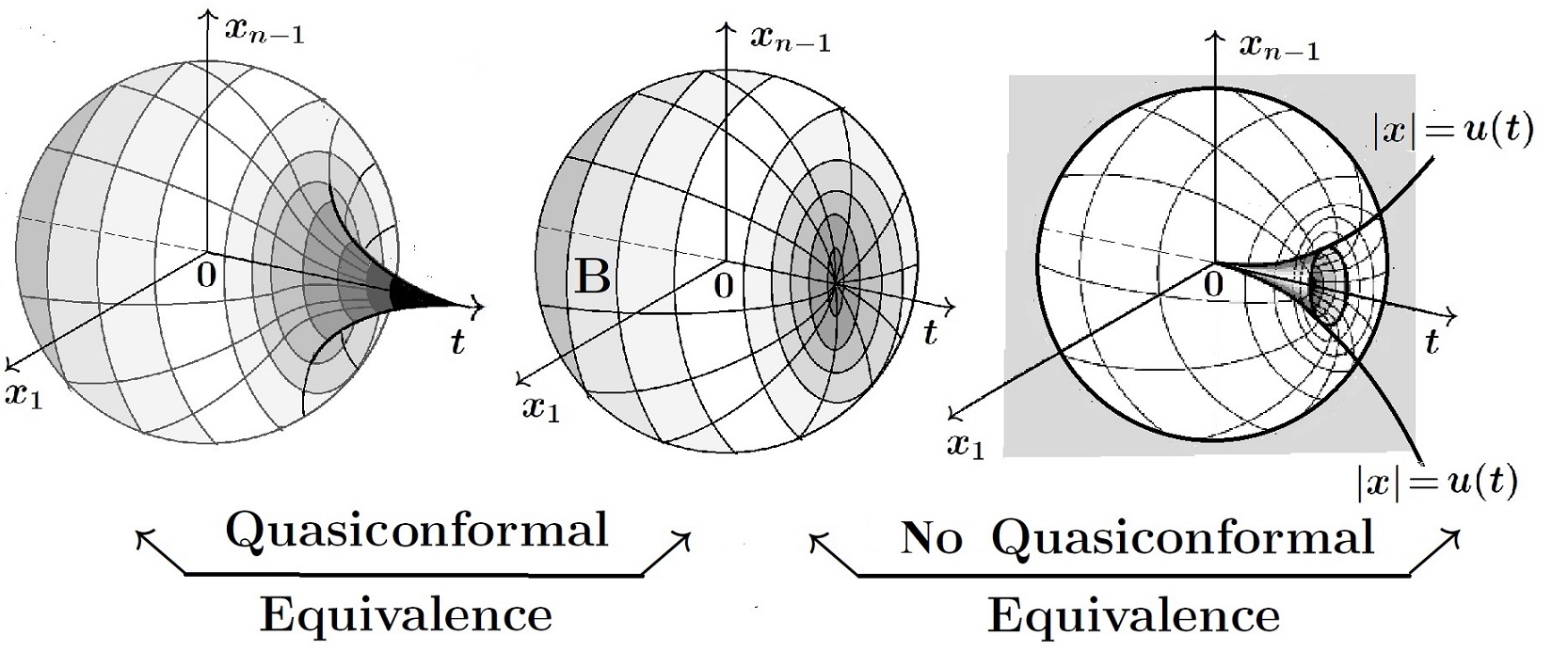}
\caption{Quasiconformal mapping can flatten the outward cusp but not the inward cusp}
\label{fig:slit}
\end{figure}

 \subsection{Inward Slit in a ball} Let us take a look at the pair $\,(\mathbb B, \mathbb B\setminus \mathbf I)\,$ of a unit ball and the ball with a slit along the line segment ${\bf I} \bydef \{(t, x) \in \R \times \R^{n-1} \colon 0\le t<1  \textnormal{ and } \abs{x}=0 \}$.

\begin{figure}[h]
\centering
\includegraphics[width=12cm]{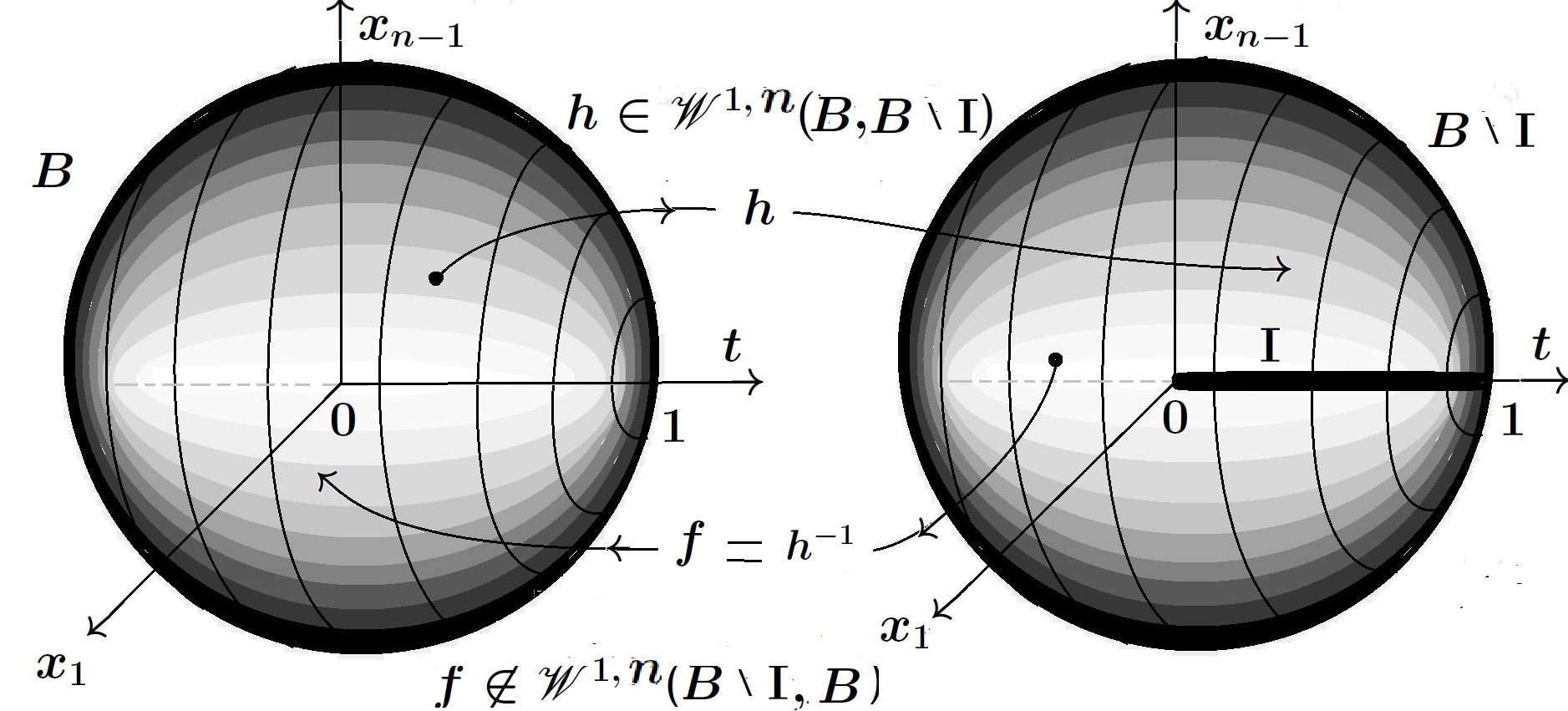}
\caption{Two domains which  are  not of  the same  bi-conformal energy type.}
\label{fig:slit}
\end{figure}
We have already mentioned that there exists a Lipschitz homeomorphism $\,h :\mathbb B \onto \mathbb B \setminus \mathbf I\,$; in particular,  $\,h \in \mathscr W^{1,n}(\mathbb B, \mathbb B \setminus \mathbf I)\,$. The question arises whether there exists a homeomorphism $\,h :\mathbb B \onto \mathbb B \setminus \mathbf I\,$ of finite conformal energy whose inverse $\,f=h^{-1} :  \mathbb B \setminus {\bf I} \onto \mathbb B\,$ also has finite conformal energy. Our next result answers this question in the negative.
\begin{theorem}\label{thm:slitbiconf}
In dimension $\,n\ge 3\,$ the domains $\,\mathbb B\,$ and $\,\mathbb B \setminus \mathbf I\,$ are not of the same bi-conformal energy type; that is, there is no homeomorphism  $ \,h : \mathbb B \onto \mathbb B\setminus {\bf I} $ of finite bi-conformal energy.
\end{theorem}

On one hand  we have:
\begin{example}\label{ex:slit}
There is a homeomorphism $f \colon  \mathbb B\setminus {\bf I}  \onto \mathbb B$ of finite conformal energy such that $h=f^{-1} \in \W^{1,\,p} (\BB, \R^n)$ for all exponents $p<n$.
\end{example}
On the other hand,  Theorem~\ref{thm:slitbiconf}  is a special case of the following.
\begin{theorem}\label{thm:cut}
 If $\ p > n-1 \geqslant 2\,$ then there is no homeomorphism  $h \colon \mathbb B \onto \mathbb B\setminus {\bf I}  $ of finite conformal energy whose inverse $h^{-1}=f\in \W^{1,p} (\mathbb B\setminus {\bf I}, \R^n)$.
\end{theorem}

The lower bound for the Sobolev  exponent in this theorem is essentially sharp; precisely, we have

\begin{theorem}\label{thm:sharp}
 For every  $\, p < n-1\,$ there is a  homeomorphism  $h \colon \mathbb B \onto \mathbb B\setminus {\bf I}  $ of finite conformal energy  whose inverse $f= h^{-1} \in \W^{1,p} (\mathbb B\setminus {\bf I}, \R^n)$.
\end{theorem}

The borderline case $\,p = n-1\,$ remains open.

\subsection{Main Result}

 Our central question is when the unit ball and the ball with a model inward cusp ${\bf S}_u$ are of the same bi-conformal energy type. Let $\,h : \mathbb B \onto \mathbb B_u\,$ be a deformation of bi-conformal energy.  To predict what  cusps ${\bf S}_u$ can be created it is natural to combine the estimates of the   modulus of continuity of  $\,h\,$ near $0$ with those  for the inverse deformation $\,f = h^{-1} : \mathbb B \setminus \mathbf I \onto \mathbb B\,$. From this point of view, deformations of bi-conformal energy are very different from  quasiconformal mappings. The latter behave like radial stretchings/squeezing; a poor  modulus of continuity is always balanced by a better modulus of continuity of its  inverse. Surprisingly,  a deformation of bi-conformal energy and its inverse may exhibit  the same optimal modulus of continuity ~\cite{IOZmod}, locally at a given point.

Let us invoke  an estimate of the modulus of continuity for  homeomorphisms $h \colon \X \onto \Y$ in $\W^{1,n} (\X, \R^n)$.

\begin{equation}\label{eq:modcont}
\abs{h(x_1)-h(x_2)}^n \le C_n\,\left(\int_{2 \mathbf B} |Dh|^n\,\right)\,  \log^{-1} \left( e + \frac{\textnormal{diam} \mathbf B}{\abs{x_1-x_2}}\right)
\end{equation}
where $ x_1, x_2 \in \mathbf B \bydef B(x_\circ, R) \subset B(x_\circ , 2R) \bydef 2 \mathbf B  \Subset \X$.  \\

Returning to our mapping $h \colon \mathbb B \onto \mathbb B_u  $ and its inverse $\,f : \mathbb B_u \onto \mathbb B $, it turns out that both mappings extend continuously up to the boundary.

\begin{theorem}\label{global}
Let $h \colon \mathbb B\onto\mathbb B_u$ be a homeomorphism of  bi-conformal energy. Then $h$ admits a homeomorphic extension to the boundary, again denoted by $h\colon \overline{\mathbb B}\onto\overline{\mathbb B_u}$.
\end{theorem}
The existence of such an extension is known~\cite[Corollary 1.1]{IO} if the reference and deformed configurations have locally quasiconformally flat boundaries, see Definition~\ref{quasiflat}. Obviously,  $\D\mathbb B_u$ is not locally quasiconformally flat.

Applying the estimates in ~\eqref{eq:modcont}  would give us a nonexistence  of a deformation of bi-conformal energy from $\mathbb B$ onto $ \mathbb B_u$  with 
 $u(t)={\exp^{-1} (\exp^{\alpha} (1/t))} $, where $\alpha >n$ (applied to both $h$ and $f$ on the boundaries). This seemingly natural approach does not lead to a sharp result. Creating and flatting  cusp  singularities through mappings of bi-conformal  energy is in a whole different scale.
\begin{theorem}
Let $n \ge 3$ and  \[ u(t)=\frac{e}{\exp \left(\frac{1}{t}\right)^\alpha} \, \quad  \textnormal{for } 0 \le t \le 1\, ,  \;  \textnormal{ where } \alpha >0\, . \] Then the domains $ \mathbb B_u$  and $\mathbb B$  are bi-conformally equivalent if and only if $\alpha <n$.
\end{theorem}
Even more,
\begin{theorem}[Main Theorem] \label{thm:cusp}
Let  $n \ge 3$ and  \[ u(t)=\frac{e}{\exp \left(\frac{1}{t}\right)^\alpha} \, \quad \textnormal{for } 0 \le t \le 1\, ,  \;    \textnormal{where } \alpha >0\, . \]
If $\alpha \ge n$ then there is no   homeomorphism  $h \colon \mathbb B \onto \mathbb B_u $ with finite conformal energy whose inverse $h^{-1}=f\in \W^{1,p} (\mathbb B_u , \R^n)$, $p>n-1$. If $\alpha <n$, then  there exists a  homeomorphism  $h \colon \mathbb B \onto \mathbb B_u  $ with finite conformal energy such that  $f$ is Lipschitz.
\end{theorem}

\section{Prerequisites}\label{Section2}
Our notation is fairly standard. Throughout the paper $\BB$ denotes the unit ball in $\rn$. We write  $C,C_{1},C_{2},...$ as generic positive constants. These constants may change even in a single string of estimates. The dependence of constant on a parameter $p$ is expressed by the notation $C=C(p)=C_p$ if needed.

We will appeal to  the  Sobolev embedding on spheres, see \cite[Lemma 2.19]{HKb}.
\begin{lemma}\label{Morryonsp}
Let $h \colon \BB\rightarrow \R^n$ be a continuous mapping in the Sobolev class  $\W^{1,p} (\BB, \R^n)$, for some $p>n-1$. Then for almost every $0<t<1$ and every $x,y\in \partial\BB(0,t) = \mathbb S_t$, we have
\begin{equation}
|h(x)-h(y)|\le C\, t^{1-\frac{n-1}{p}}\lf(\int_{\mathbb S_t}|Dh(x)|^p\, \dtext x\r)^{\frac{1}{p}}.\nonumber
\end{equation}
Here the constant $C$ depend only on $n$ and $p$.
\end{lemma}

It is relatively easy to conclude from this estimate that a $\W^{1,p}$-homeomorphism when $p>n-1$ is differentiable almost everywhere. It also follows that a homeomorphism   $h\colon \X \onto \Y$ in the Sobolev class $\W^{1,n} (\X, \R^n)$ satisfies Lusin's condition $(N)$. This simply means, by definition, that $\abs{h(E)}=0$ whenever $\abs{E}=0$.

\begin{lemma}\label{W1n}
Let $\X,\Y$ be domains in $\R^n$ and $h \colon \X \onto \Y $ be a homeomorphism in the Sobolev class $\W^{1,n}(\X,\Y)$. Then $h$ is differentiable almost everywhere and  satisfies Lusin's condition $(N)$.
\end{lemma}

Due to Lusin's condition $(N)$  we have the following version of change of variables formula, see e.g.~\cite[Theorem 6.3.2]{IMb} or~\cite[Corollary A.36]{HKb}.

\begin{lemma}\label{co-area}
Let $h \colon \X \onto \Y$ be a homeomorphism in the Sobolev class $\W^{1,n} (\X,\rn)$.  If $\eta$ is a nonnegative Borel measurable function on $\rn$ and  $A$  a Borel measurable set in $\X$, then we have
\begin{equation}\label{eq:cvf}
\int_{A}\eta(h(x))|J_h(x)|\, \dtext x=\int_{h(A)}\eta(y)\, \dtext y\, .
\end{equation}
\end{lemma}

Next, we recall a well-known fact that a function in the Sobolev class $\W^{1,p} (\X, \R)$, $\X \subset \R^n$, is locally H\"older continuous with exponent $1-p/n$, provided $p>n$. More precisely, we have the following oscillation lemma.
\begin{lemma}\label{lem:holder}
Let $u \in \W^{1,p} (\X, \R)$ where $\X \subset \R^n$ and $p>n$. Then
\[ \abs{u(x)-u(y)} \le C \, r^{1-\frac{p}{n}} \left( \int_{\mathbb B_r} \abs{\nabla u}^p\right)^\frac{1}{p}\]
for every $x,y \in \mathbb B_r= \mathbb B (z,r) \subset \X$.
\end{lemma}

We will employ a higher dimension version of the classical Jordan curve theorem due to Brouwer~\cite{Br}.
\begin{lemma}\label{lem:jbs} (Jordan-Brouwer separation theorem)
A topological $(n-1)$-sphere $\mathcal S$ disconnects $\R^n$ into two components the bounded component denoted by $\mathcal S_\circ$ and the unbounded component denoted by $\mathcal S_\infty$. Their common boundary is $ \overline{\mathcal S_\circ} \cap \overline{\mathcal S_\infty}=  \mathcal S$.
\end{lemma}

Theorem~\ref{global} claims that a homeomorphism $h \colon \BB \onto \BB_u $ of bi-conformal energy can be extended as a homeomorphism from $\overline{\BB}$ onto $\overline{\BB_u}$.
The existence of continuous extension of such a homeomorphism follows from the following result, see~ \cite[Theorem 1.3]{IO}.
\begin{lemma}\label{extend}
Let $\X$ and $\Y$ be bounded domains of finite connectivity. Suppose $\D\X$ is locally quasiconformally flat and $\D\Y$ is a neighborhood retract. Then every homeomorphism $h\colon \X\onto\Y$ in the class $h\in\W^{1,n}(\X, \Y)$ extends to a continuous map $h\colon  \overline\X\to\overline\Y$.
\end{lemma}

The assumed boundary regularities are defined as follows.

\begin{definition}\label{retract}
The boundary $\partial\Y$ is a neighborhood retract, if there is a neighborhood $\U\subset\rr^n$ of $\partial\Y$ and a continuous map $\chi:\U\rightarrow\partial\Y$ which is an identity on $\partial\Y$.
\end{definition}
\begin{definition}\label{quasiflat}
The boundary $\D\X$ is said to be locally quasiconformally flat if every point in $\D\X$ has a neighborhood $\U\subset\rr^n$ and a homeomorphism $g:\U\cap\overline{\X}\onto  \BB \cap ( \R^{n-1}\times \R^+)$ which is quasiconformal on $\U\cap\X$; see \cite{Vaib}.
\end{definition}
Recall that $\R^+=[0, \infty)$.
It is also known that a mapping of bi-conformal energy between domains with  locally quasiconformally flat boundaries has a homeomorphic extension up to the boundary, see~ \cite[Corollary 1.1]{IO}. Note that $\D\mathbb B_u$ is not locally quasiconformally flat and this result does not apply in our case.

Nevertheless, Lemma~\ref{extend} tells us that $h$ extends as a continuous mapping $h \colon \overline{\mathbb B} \to \overline{\mathbb B_u}$. Since $h(\overline{\mathbb B} )$ is a compact subset of $\overline{\mathbb B_u}$, it follows that $h$ takes $\overline{\mathbb B}$ onto $\overline{\mathbb B_u}$. Second, it is  a topological fact~\cite{Fl} that such a continuous extension is a monotone mapping $h \colon \overline{\mathbb B} \onto \overline{\mathbb B_u}$:
\begin{proposition}\label{pro:mono}
Suppose that there is a continuous extension  $G \colon \overline{\mathbb B} \onto \overline{\mathbb B}$ of a homeomorphism  $g  \colon {\mathbb B} \onto {\mathbb B}$. Then $G \colon \partial \mathbb B \onto \partial \mathbb B $ is monotone.
\end{proposition}
 By the definition, monotonicity, the concept of  Morrey~\cite{Mor}, simply means that for a continuous $h \colon \overline{\X} \to \overline{\Y}$ the preimage $h^{-1} (y_\circ )$ of a point   $y_\circ \in \overline{\Y}$ is a connected set in $\overline{\X}$.
It is worth noting that the converse statement of Proposition~\ref{pro:mono} is also valid when $n=2,3$.  Such an elegant characterization of monotone
mappings of a $2$-sphere onto itself was obtained by Floyd and Fort~\cite{FF}.

In the next lemmas we will analyze the boundary behavior of  continuous extension of  homeomorphism $h \colon \mathbb B \onto \mathbb B_u$ with finite conformal energy. 

\begin{lemma}\label{lem:mono}
 Suppose a homeomorphism $h \colon \mathbb B \onto \mathbb B_u$ lies in the Sobolev class $\W^{1,n} (\mathbb B, \mathbb R^n)$. Then for every $x\in\partial {\BB_u}$ the preimage $h^{-1}(x)$ is a nonempty continuum in $\partial {\mathbb B}$.
\end{lemma}


Simplifying writing we set $o\bydef (0,0,...,0)$ and $o' \bydef (1,0,...,0)$. Without loss of generality, we may assume that $h(o')=o$. For every $0<t<1$, we define
\begin{equation}
S_t\bydef \{x\in\BB_u \colon |x|=t\}\ \ {\rm and }\ \ C_t \bydef \{x\in\partial\BB_u \colon |x|=t\}.\nonumber
\end{equation}
Furthermore, let  $S'_t \bydef h^{-1}(S_t)$ and $C'_t \bydef \overline{S'_t}\cap\partial \BB$.

\begin{lemma}\label{lem:27}
Under the assumption of Lemma~\ref{lem:mono} we have $h(C'_t)=C_t$.
\end{lemma}
\begin{proof}

For every $x_t\in C_t$, there exists a sequence $\{x_{t,i}\}\subset S_t$ with $x_t=\lim\limits_{i\rightarrow\fz}x_{t,i}$ and that the corresponding sequence $\{x'_{t,i}\bydef f(x_{t,i})\}$ in $ S'_t$ is also convergent. We write $x'_t\bydef \lim\limits_{i\rightarrow\fz}x'_{t,i}$. Then since $h\colon \overline{\BB}\rightarrow\overline{\BB_u}$ is continuous we have  $h(x'_t)=x_t$. By Lemma~\ref{lem:mono}, $x'_t \in \partial \BB$ and therefore $x'_t\in C'_t$.
\end{proof}

\begin{lemma}\label{Morry}
Suppose that a homeomorphism $h \colon \mathbb B \onto \mathbb B_u$ has finite conformal energy. If the inverse mapping $f=h^{-1} \colon \BB_u\rightarrow \BB$ belongs to the Sobolev class  $\W^{1,p} (\BB_u ,
\R^n) $ for some $p>n-1$,  then for almost every $0<t<1$ and every $x'_t, y'_t\in C'_t$ we have
\begin{equation}\label{eq:sobosphere}
|x'_t-y'_t|\le C|x_t-y_t|^{1-\frac{n-1}{p}}\lf(\int_{S_t}|Df|^pdx\r)^{\frac{1}{p}}.
\end{equation}
Here  $x_t=h(x'_t)$ and $y_t=h(y'_t)$ and $C$ is a positive constant independent of $t$, $x_t$ and $y_t$.
\end{lemma}
\begin{proof}
Let $x'_t, y'_t\in C'_t$. By Lemma~\ref{lem:27}  there are  two sequences $\{x'_{t,i}\}^\infty_{i=1} $ and $ \{y'_{t,i}\}^\infty_{i=1} $ in $ S'_t$ such that
\begin{equation}
\lim_{i\rightarrow\fz}x'_{t,i}=x'_t,\ \ \lim_{i\rightarrow\fz}y'_{t,i}=y'_t\nonumber
\end{equation}
and
\begin{equation}
\lim_{i\rightarrow\fz}x_{t,i}=x_t \in C_t, ,\ \ \lim_{i\rightarrow\fz}y_{t,i}=y_t \in C_t,  \, .\nonumber
\end{equation}
Here
\begin{equation}
x_{t,i}=h(x'_{t,i}), y_{t,i}=h(y'_{t,i}), x_t=h(x'_t)\ {\rm and}\ y_t=h(y'_t).\nonumber
\end{equation}
By  the classical Sobolev embedding on sphere, Lemma \ref{Morryonsp}, we have
\begin{equation}
|x'_{t,i}-y'_{t,i}|\le C|x_{t,i}-y_{t,i}|^{1-\frac{n-1}{p}}\lf(\int_{S_t}|Df|^pdx\r)^{\frac{1}{p}}.\nonumber
\end{equation}
Passing to the limit, we obtain
\begin{equation}
|x'_t-y'_t|\le C|x_t-y_t|^{1-\frac{n-1}{p}}\lf(\int_{S_t}|Df|^pdx\r)^{\frac{1}{p}}.\nonumber
\end{equation}
\end{proof}

If $f\in \W^{1,p}(\BB_u,\R^n)$, $p>n-1$, then there is  a decreasing sequence $\{t_i\}_{i=1}^{\fz}$ with $0<t_1<1$, which converges to $0$,  and satisfies~\eqref{eq:sobosphere} and
\begin{equation}
\int_{S_{t_i}}|Df|^p\, \dtext x<\frac{1}{t_i}.\nonumber
\end{equation}
Indeed, if not, then by  Fubini's theorem  for $T\in (0,1)$ we have
\begin{equation}
\int_{\BB_u}|Df(x)|^p \, \dtext x\ge \int_0^T\int_{S_t}|Df(x)|^p\, \dtext x \dtext t\ge\int_0^T\frac{1}{t}\, \dtext t=\fz.\nonumber
\end{equation}
Without loss of generality, we may also assume that $\diam C'_{t_i}$ is decreasing with respect to $t_i$ and $\diam C'_{t_1}<\frac{1}{4}$.

According to Lemma~\ref{lem:27} and Lemma~\ref{Morry} $h \colon C'_t \onto C_t$ is a homeomorphism. Now, Jordan-Brouwer Separation Theorem, Lemma~\ref{lem:jbs}, tells us that.
\begin{lemma}\label{lem:twocomp}
Under the assumptions of Lemma~\ref{Morry} it follows that  $\partial \mathbb B \setminus C'_t$ consists of two disjoint connected open sets whose common boundary is $C'_t$.
\end{lemma}
The boundary mapping $h \colon \partial \mathbb B \onto \mathbb B_u$ is monotone. More, however, can be sad about the preimage of the singular point. 
\begin{lemma}\label{single}
Under the assumptions of Lemma~\ref{Morry} we have $h^{-1}(o)=o'$.
\end{lemma}
\begin{proof}
According to Lemma~\ref{lem:twocomp}, $\partial \mathbb B \setminus C'_t$ consists of two disjoint connected open sets whose common boundary is $C'_t$. 
We denote the one with smaller diameter by $\U_t$.  Now, for  $0<t<\tau<t_1$, we have $ U_t\subset U_{\tau}$ and we denote  $U_\circ \bydef \lim\limits_{t\rightarrow 0}\overline{U_t}$. Combining this with continuity of $h \colon \overline{\BB} \onto \overline{\BB_u}$, we obtain
\begin{equation}\label{eq:22}
h(U_\circ )=\lim_{t\rightarrow0}h(\overline{U_t})
\end{equation}
Since $C'_t \subset \overline{U_t}$, $\lim\limits_{t\rightarrow0}C_t=o$ and $h(C'_t)=C_t$, see Lemma~\ref{lem:27}, we have $o \in  h(U_\circ ) \subset h(\overline{U_t})$  for every $0<t<t_1$. By Lemma~\ref{lem:mono} $h^{-1}(o)$ is a continuum, we obtain that $h^{-1}(o)\subset \overline{U_t}$ for every $0<t<t_1$. By Lemma \ref{Morry}, $\diam C'_t$ will converge to $0$ as $t$ goes to $0$. Therefore, also the diameter of $\overline{U_t}$ approaches $0$. Hence  $h^{-1}(o)=o'$.
\end{proof}

We will close  this section to give a precise modulus of continuity estimate for a homeomorphism  $h \colon {\BB} \onto {\BB_u}$ with finite conformal energy. Recall that such a homeomorphism has a continuous extension up to the boundary. Furthermore, the boundary mapping  $h \colon \partial {\BB} \onto  \partial {\BB_u}$ is monotone in the sense of Morrey, see Lemma~\ref{lem:mono}. Monotone mappings enjoy a property which is  commonly known in literature also as monotonicity. This notion goes back to
H. Lebesgue~\cite{Le} in 1907.  To avoid confusion, in the following definition we use the term {monotone in the sense of Lebesgue}.

\begin{definition} \label{def:monoleb}
Let $\X$ be an open subset of $\R^n$. A continuous function $h\colon \overline{\X} \to \R^n$ is  \emph{monotone in the sense of Lebesgue} if for every compact set $K\subset \overline{\X}$ we have
\begin{equation}\label{diamdef}
\diam h(K) = \diam h(\partial K).
\end{equation}
Note that for real-valued functions~\eqref{diamdef} can be stated as
\[
\min_{K}h=\min_{\partial K} h \le \max_{\partial K} h = \max_{K}h.
\]
\end{definition}

\begin{remark}
A folding map is a characteristic example of  continuous nonmonotone mapping which is monotone in the sense of Lebesgue.
\end{remark}

\begin{lemma}\label{modulus}
Let  $h \colon \BB\rightarrow\BB_u$ be a homeomorphism  with finite conformal energy. If $h(o')=o$, then  there exists an increasing function $\varepsilon  \colon [0,1) \to [0, \infty)$ with $\lim\limits_{t\to 0+} \varepsilon (t)=0$ such that for $x'\in \overline{ \BB}$ with $0<|x'-o'|<1$ we have
\begin{equation}\label{moduin}
|h(x')-h(o')|\le \frac{\varepsilon (|x'-o'|)}{\log^{\frac{1}{n}}\left( \frac{1}{|x'-o'|}\right)}  \, .
\end{equation}

\end{lemma}
\begin{proof}
Set
\begin{equation}
\mathcal S_t\bydef \partial B(o',t)\cap\overline{B}(0,1),\nonumber
\end{equation}
and
\begin{equation}
\osc (h, \mathcal S_t) \bydef \max_{x'_t,y'_t\in\mathcal S_t}|h(x'_t)-h(y'_t)|.\nonumber
\end{equation}
Since $h \colon \overline{\BB} \onto \overline{\BB_u}$ is continuous and belongs to the Sobolev class $\W^{1,n} (\BB, \R^n)$ applying a slightly modified version of Sobolev embedding on sphere, Lemma~\ref{Morryonsp} for almost every $0<t<1$ we have
\begin{equation}\label{equation1}
\lf( \osc (h, \mathcal S_t) \r)^n\le Ct\int_{\mathcal S_t}|Dh(x)|^n\, \dtext x.
\end{equation}
Here $C$ is a positive constant, independent of $t$. Fix $x'\in \BB$ such that $\tau\bydef |x'-o'|<1$. We write
\[\mathcal B (o', t) \bydef \BB\cap \BB(o',t) \quad \textnormal{for } 0<t<1 \, . \]
Choose $t\in [{\tau } ,\sqrt{\tau}]$. Then
\[\osc (h, \overline{\mathcal B (o', {\tau})}  ) \le \osc (h, \overline{\mathcal B (o', t )}) \le \osc (h, \partial \overline{\mathcal B (o', t )}) \]
where the latter inequality follows from the  fact that $h$ is monotone in the sense of Lebesgue. By the geometry of $\overline{\BB_u}$, we have
\[ \osc (h, \partial \overline{\mathcal B (o', t )}) = \osc (h,  {\mathcal S_t)})  \]
Combining this with~\eqref{equation1} for almost every $ t \in [\tau, \sqrt{\tau}] $ we have
\[  \frac{ \lf( \osc (h,      \overline{\mathcal B (o', {\tau})}        \r)^n }{t} \le C \,   \int_{\mathcal S_t}|Dh(x)|^n\, \dtext x \, .    \]
Integrating this from $\tau$ to $\sqrt{\tau}$ with respect to the variable $t$, the claimed inequality~\eqref{moduin} follows with
 \begin{equation}
\varepsilon (\tau ) =C\cdot \left(\int_{ \mathcal B(o', \sqrt{\tau} )}|Dh(x)|^ndx\r)^{\frac{1}{n}}\, ,   \qquad \tau = |x'-o'| \, .
\end{equation}

\end{proof}

\section{Proof of Theorem \ref{thm1.2}}
Theorem~\ref{thm1.2} is known among the experts in the field and  easily follows  combining a few results in the literature.  We mainly provide a proof for the convenience of the reader.

\begin{proof}
First, we assume that $K_I(\cdot, h)\in \cl^{\, 1}(\X)$. Then, Theorem in~\cite{AIMO} states that a homeomorphism $h \in \W^{1,n} (\X, \R^n)$ satisfies the claimed identity~\eqref{eq:identity} if  $h$ has a finite (outer) distortion; that is, there is a function $1\le K_{_O} (x) < \infty$ such that
\begin{equation}\label{eq:outer}\abs{Dh(x)}^n \le K_{_O} (x) \, J_h (x) \qquad \textnormal{for almost every } x \in \X \, . \end{equation}
The proof, however, only uses a consequence of~\eqref{eq:outer} the finite inner inequality~\eqref{eq:innerdistortion}, see \cite[(9.10)]{AIMO}.

Second, we assume that $h \in \W^{1,n} (\X, \R^n)$ and $f=h^{-1}\in \W^{1,n} (\Y, \R^n)$. Then
\begin{equation}\label{eq:idenk_i}
K_{_I} (x,h) = \abs{Df \big( h(x)\big)}^n J_h(x) \quad \textnormal{ a.e. } x \in \X  \, .
\end{equation}
Indeed, by  Lemma~\ref{W1n} both $h $ and $f$ are differentiable almost everywhere. Now, the identity $(f \circ h) (x)=x$, after differentiation,  implies that
\begin{equation}\label{eq:blah10} Df (h(x)) Dh(x)= {\bf I} \quad \textnormal{ a.e. in } \X.
\end{equation}
Since both $h$ and $f$ satisfy Lusin's condition $(N)$; that is,  preserve sets of zero measure, see Lemma~\ref{W1n}. This shows that $J_h(x)>0$ and $J_f(y)>0$ almost everywhere again we used the fact that $h$ satisfies Lusin's condition $(N)$. Now, the formula~\eqref{eq:idenk_i} is a direct consequence of the definition of the inner distortion, Gramer's rule $Dh (x) D^\sharp h(x) = J_h (x) {\bf I}$ and~\eqref{eq:blah10}
\[K_{_I} (x,h) = \frac{|D^\sharp h(x)|^n}{|J_h(x)|^{n-1}} = \abs{(Dh (x))^{-1}}^n J_h (x) = \abs{Df \big( h(x)\big)}^n J_h(x)\, .  \]
Now the change of variables formula~\eqref{eq:cvf} gives
\[\int_{\X} K_{_I} (x,h) \, \dtext x = \int_{\Y} \abs{Df(y)}^n \, \dtext y \, .  \]
\end{proof}

\begin{proof}[Proof of Corollary \ref{quasiconformal}]
By~\cite[\S 6.4]{IMb}  for every $x\in\X$ with $J_h(x)>0$, we have
\begin{equation}\label{eq:disto}
K_{_I}^{\frac{1}{n-1}}(x, h)\le K_{_O}(x, h)\le K_{_I}^{n-1}(x, h)\, .
\end{equation}
Here $ K_{_O}(x, h)$ stands for the smallest function satisfying~\eqref{eq:outer}. Now, Corollary \ref{quasiconformal} follows immediately from~\eqref{eq:disto}.
\end{proof}

\section{Proof of Theorem \ref{global}}
\begin{proof}[Proof of Theorem \ref{global}]
By Lemma~\ref{extend} a homeomorphism $h \colon \overline{\mathbb B} \to \overline{\mathbb B_u}$ with finite conformal energy extends as a continuous mapping $h \colon \overline\BB\to\overline{\BB_u}$.  Since $h(\overline{\mathbb B} )$ is a compact subset of $\overline{\mathbb B_u}$, it follows that $h \colon  \overline\BB\onto\overline{\BB_u}$.
Furthermore, by Lemma~\ref{lem:mono} the boundary map $h \colon \partial \BB\onto  \partial {\BB_u} $ is monotone.

Now, we need to show that the boundary mapping is injective. We again use  the notation $o= (0,0,...,0)$ and $o' = (1,0,...,0)$ and assume, without loss of generality, that $h(o')=o$. First, $h^{-1} (o) =o'$ by Lemma~\ref{single}. Second let $y \in \partial \BB_u \setminus \{o\}$.  Choosing  $0<r_y< |y-o|$, then $\BB(y,r_y)\cap{\BB_u}$ is locally quasiconformally flat. By Lemma~\ref{extend}, the homeomorphism $f\colon \BB(y,r_y)\cap {\BB_u} \onto f( \BB(y,r_y)\cap {\BB_u} )$ has a  continuous extension $f\colon \overline{\BB(y,r_y)\cap {\BB_u} } \onto \overline{f( \BB(y,r_y)\cap {\BB_u} )}$.  Therefore, $h^{-1}(y)=f(y)$ is a single point. Now  we know that $h\colon \overline\BB\onto\overline{\BB_u}$ is a continuous bijection, and therefore it is a homeomorphism.
\end{proof}

\section{Construction of Example~\ref{ex:slit}}
Here we show that there exists a homeomorphism from $\BB \setminus {\bf I}$ onto $\BB$ with finite conformal energy actually Lipschitz continuous whose inverse lies in $\W^{1,p} (\BB, \R^n)$ for every $p<n$. To simplify our construction, we may and do replace $\mathbb B$  by a bi-Lipschitz equivalent domain; namely,
\[\Y = \{  (s,y)\in \R \times \R^{n-1} \colon  \abs{y} < 1 \textnormal{ and } -1<s< \abs{y} \}\]
As for the reference configuration we replace $\mathbb B \setminus {\bf}$ by cylinder ${\bf C}=(-1,1) \times \mathbb B^{n-1}$ with the line segment {\bf I} removed from it. Consider the Lipschitz homeomorphism  $h \colon {\bf C} \setminus {\bf I} \onto \mathbb Y$ defined by the rule
\begin{equation}
h(t,x) =  \begin{cases} (t \abs{x} ,\,x) \quad & \textnormal{ for } t>0 \\
(t,x)&  \textnormal{ for } t<0 \, . \end{cases}
\end{equation}
Its inverse mapping $f \colon \mathbb Y \onto  {\bf C} \setminus {\bf I}  $ takes the form
\begin{equation}
f(s,y) =  \begin{cases} \left( \frac{s}{\abs{y}}  , \, y\right)  \quad & \textnormal{ for } s\ge 0 \\
(s,y)&  \textnormal{ for } s<0 \, . \end{cases}
\end{equation}
It is easy to see that 
\[\abs{Df(s,y)} \le \frac{C_n}{\abs{y}} \]
Therefore,
\[\int_\Y \abs{Df}^p < \infty \qquad \textnormal{ for every } 1\le p<n  \]
as desired.

\section{Proof of Theorem \ref{thm:cut}}
\subsection{The nonexistence part of Theorem \ref{thm:cut}} First, we will prove the nonexistence part of Theorem \ref{thm:cut}.
\begin{thm}\label{nonexistence}
If $p>n-1$, then there is no homeomorphism $h \colon \BB\onto\BB\setminus\bf I$ with $h\in \W^{1,n}(\BB,\BB\setminus\bf I)$ whose inverse $f=h^{-1}\in \W^{1,p}(\BB\setminus\bf I,\BB)$.
\end{thm}
\begin{proof}
Suppose to the contrary that there is a homeomorphism $h \colon \BB\onto \BB\setminus\bf I$ in the Sobolev class $\W^{1,n}(\BB,\BB\setminus\bf I)$ such that  $f\in \W^{1,p}(\BB\setminus\bf I,\BB)$. Since $\D(\BB\setminus\bf I)$ is a neighborhood retract, Lemma~\ref{extend} tells us that the homeomorphism $h \colon \BB\onto\BB\setminus\bf I$  extends as  a continuous  mapping $h \colon \overline\BB\onto\overline{\BB }$. We denote
\[S_t = \partial B_t \setminus \{x_t\}   \qquad \textnormal{where } x_t \bydef  (t, 0, \dots , 0) \, .  \]
 Here $B_t=B(0,t)$.  Fubini's theorem implies that for almost every $t\in(0, 1)$, $f\big|_{S_t }\in\W^{1,p}(S_t , \R^n)$. Since $p>n-1$ and $n \ge 3$, the possible singularity of $f$ at $x_t$ is removable. For such  $t$, applying  Lemma~\ref{lem:holder},  $f\big|_{S_t }$ extends as a homeomorphism $f \colon \overline{S_t} \onto f(\overline{S_t})$. Write $x'_t= f(x_t)$. Now, Jordan-Brouwer Separation Theorem, Lemma~\ref{lem:jbs}, tells us that  $\R^n \setminus f(\overline{S_t})$ consists of two disjoint connected open sets whose common boundary is $f(\overline{S_t})$.  Let us denote the bounded one by $U_t$. Note that $U_t \subset \BB$ and $\overline{U_t} \cap \partial \BB =\{x'_t\} $.  Since for almost every $t<s\in (0,1)$ we have $B_t \setminus {\bf I} \subset B_s \setminus {\bf I} $ then $U_t=h^{-1} (B_t \setminus {\bf I}) \subset h^{-1} (B_s \setminus {\bf I})=U_s$. Now comes an elementary topological fact; given two domains $U \subset V \subset \mathbb B$ such that $\overline{U} \cap \partial \mathbb B = \{x_\nu\}$ and  $\overline{V} \cap \partial \mathbb B = \{x_\mu\}$, then $x_\nu = x_\mu$.
  
Now, we have $x_s'=x'_t$. This, however, is impossible since  $h(x_s')=(s,0, \dots , 0)$ and $h(x_t')=(t,0, \dots , 0)$.

\end{proof}

\subsection{The existence part of Theorem \ref{thm:cut}}
Here we verify  the existence part of Theorem \ref{thm:cut}. Namely,
\begin{thm}\label{thm:62}
There exists a Lipschitz homeomorphism $h \colon \BB\rightarrow\BB\setminus\bf I$ whose  inverse $f\in \W^{1,p}(\BB\setminus\bf I,\BB)$ for every $1\le p<n-1$.
\end{thm}
\begin{proof}
We shall view $\R^n$ as 
\[\R^n = \R \times \R^{n-1}= \{(t,x) \colon  t\in \R \, , \; x \in \R^{n-1}\}  \, . \]
To simplify our construction, we may and do replace $\mathbb B$ by a bi-Lipschitz equivalent domain; namely $\X=\X_-\cup \X_+$, where
\[\X_- = \{ (t,x) \colon -1<t<0 \textnormal{ and } \abs{x}<1  \}\]
and
\[\X_+ = \{ (t,x) \colon 0\le t<1 \textnormal{ and } \frac{t}{2} < \abs{x}<1  \} \, . \]

As for the reference configuration we consider  $\Y = \Y_+ \cup \Y_-$ where $\Y_-$ is the open unit cylinder
\[\Y_- = \{ (s,y) \colon -1<s<0 \textnormal{ and } \abs{y}<1 \}\]
and
\[ \Y_+= \{(s,y) \colon 0 \le s <1 \textnormal{ and } 0 < \abs{y}<1\}  \, . \]

\begin{figure}[h]
\centering
\includegraphics[width=12.7cm]{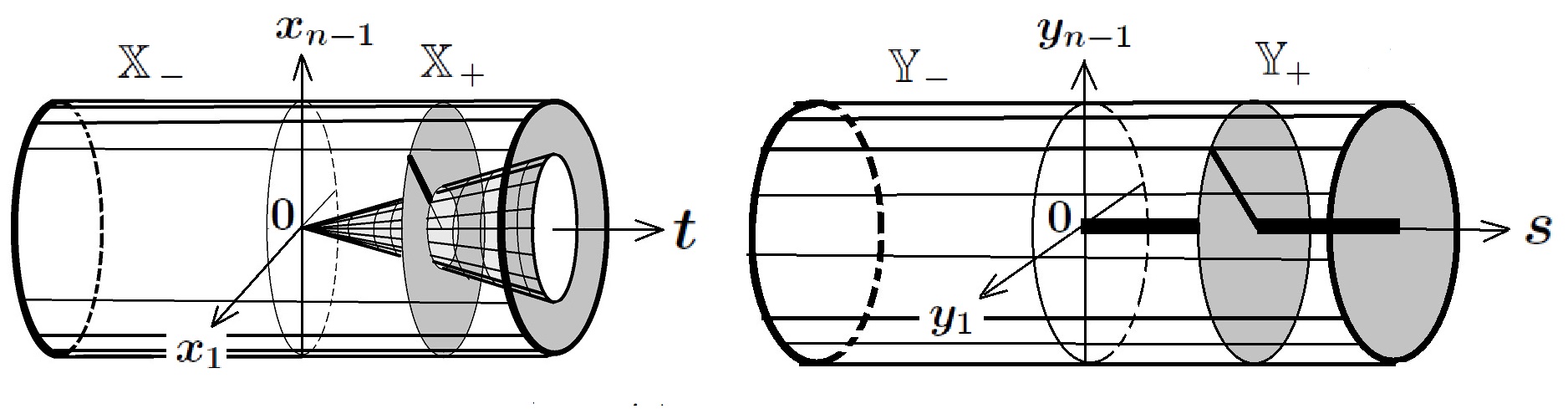}
\caption{The domains $\X$ and $\Y$.}
\label{fig:ex}
\end{figure}

We define a Lipschitz map $h \colon \X \onto \Y$ by the rule
\[
h(t,x)= \begin{cases} (t,x) & \textnormal{in } \X_- \\
\left(t, \left[ \frac{2\abs{x}}{2-t}- \frac{t}{2-t}\right] \frac{x}{\abs{x}} \right) \quad & \textnormal{in } \X_+ \end{cases}
\]
Then the inverse map $f=h^{-1} \colon \Y \onto \X$ takes the form
\[
f(s,y)= \begin{cases} (s,y) & \textnormal{in } \Y_- \\
\left(s, \left[ \frac{2-s}{2} \abs{y}+ \frac{s}{2}\right] \frac{y}{\abs{y}} \right) \quad & \textnormal{in } \Y_+ \end{cases}
\]
It is the identity map on $\Y_-$ while on $\Y_+$ we write it as
\[f(s,y)= \left(s, \frac{2-s}{2}y\right)  + \left(0, \, \frac{sy}{2\abs{y}}\right)\]
where the first term is $\mathscr C^\infty$-smooth. It is now easy to verify the estimate
\[  \abs{Df(s,y)}  \le C \cdot \left(1+ \frac{s}{\abs{y}} \right)\]
where $\abs{s}<1$ and $y \in \R^{n-1}$, $0< \abs{y}<1$. Hence
\[\int_{\Y_+} \abs{Df}^p<\infty \quad \textnormal{for every } 1 \le p < n-1\]
as desired.
\end{proof}

\section{Proof of Theorem \ref{thm:cusp}}
\subsection{The nonexistence part of Theorem \ref{thm:cusp}} Here we give a proof of the nonexistence part of Theorem \ref{thm:cusp}. Recall the statement  for the convenience of the reader.
\begin{thm}\label{nonexistence}
Let $\alpha\ge n$ and $p>n-1$ be fixed and $u(t)=\frac{e}{\exp\lf(\frac{1}{t}\r)^{\alpha}}$. Then there does not exists a homeomorphism $h:\BB\to\BB_u$ with $h\in\W^{1,n}(\BB,\BB_u)$ and $h^{-1}\in\W^{1,p}(\BB_u,\BB)$.
\end{thm}
\begin{proof}
Fix  $\alpha\ge n$ and $p>n-1$. Suppose to the contrary that there exists a homeomorphism $h \colon \BB \onto \BB_u$  with finite conformal energy such that  its inverse $f\in W^{1,p}(\BB_u, \R^n)$. According to  Lemma \ref{extend}, $h$  extends as  a continuous  mapping $h \colon \overline{\BB}\onto \overline{\BB_u}$. Furthermore, by Lemma~\ref{lem:mono} the boundary mapping $h \colon \partial \BB \onto \partial \BB_u$ is monotone.

We follow the notation introduced in Section~\ref{Section2} and set $o= (0,0,...,0)$ and $o' = (1,0,...,0)$. We may and do assume that $h(o')=o$. Moreover, for every $0<t<1$, 
\begin{equation}
S_t = \{x\in\BB_u \colon |x|=t\}\ \ {\rm and }\ \ C_t =  \{x\in\partial\BB_u \colon |x|=t\}.\nonumber
\end{equation}
and
\[  S'_t= h^{-1}(S_t) \quad \textnormal{and} \quad C'_t = \overline{S'_t}\cap\partial \BB \, . \]
Lemma~\ref{lem:twocomp} tells us that $C'_t$ divides $\partial \BB$ into two disjoint components. We denote the component which  contains $o'$ by $U'_t$. Accordingly, we also have
\begin{equation}\label{eq:apuo}
\partial U'_t = C'_t \, .
\end{equation}
 Since
\begin{equation}
\int_{\BB_u}|Df(x)|^pdx<\fz,\nonumber
\end{equation}
 there exists a decreasing sequence $\{t_i\}$, which converges to $0$ and satisfies
\begin{equation}\label{1/t}
\int_{S_{t_i}}|Df(x)|^pdx<\frac{1}{t_i} \, .
\end{equation}
Indeed, by Fubini's theorem we have 
\[\int_0^1 \int_{S_t} \abs{Df(x)}^p \, \dtext x < \infty \, , \quad \textnormal{hence } \liminf_{t\to 0} t \int_{S_t} \abs{Df(x)}^p =0\]

Now, by Lemma~\ref{lem:27} we have $h(C_t')=C_t$. Combining this with  Lemma \ref{Morry}  for every $x'_{t_i}, y'_{t_i} \in C'_{t_i}$ we have
 \begin{eqnarray}\label{equa1}
\diam C'_{t_i} &\le& C\cdot \abs{x_{t_i}-y_{t_i}}^{1-\frac{n-1}{p}}\lf(\int_{S_{t_i}}|Df(x)|^pdx\r)^{\frac{1}{p}} \nonumber \\
&\le& C\cdot \big(2\, u(t_i)\big)^{1-\frac{n-1}{p}}\lf(\int_{S_{t_i}}|Df(x)|^pdx\r)^{\frac{1}{p}} \\
                     &\le& C\cdot \big (u(t_i)\big)^{1-\frac{n-1}{p}}\lf(\frac{1}{t_i}\r)^{\frac{1}{p}}\nonumber
 \end{eqnarray}
Here $u(t)=\frac{e}{\exp\lf(\frac{1}{t}\r)^\alpha}$. Especially, this shows that $\diam (C'_{t_i}) \to 0$ as $i \to \infty$ and, therefore,  $U'_{t_i}$ lies on the half sphere $\partial \mathbb B_+$. We now appeal to the geometric fact if $x,a\in U$, then $\abs{x-a} \le \diam \partial U$. Now, by~\eqref{eq:apuo} we choose $x'_{t_i} \in C'_{t_i}$ such that
\begin{equation}\label{equa}
|x'_{t_i}-o'|\le \diam C'_{t_i} \, .
\end{equation}
According to  Lemma \ref{modulus} we obtain
 \begin{equation}\label{equa2}
t_i\le \abs{h(x'_{t_i}) -o}\le {\varepsilon (t_i)} \, {\log^{-\frac{1}{n}}\frac{1}{|x'_{t_i}-o'|}},
 \end{equation}
where $\varepsilon (t)$ is a positive function which converges to $0$ as $t$ goes to $0$.
Combining this with~\eqref{equa} we have
\begin{equation}\label{eq:apu5}t_i \le {\varepsilon (t_i)}\, {\log^{-\frac{1}{n}}\frac{1}{\diam C'_{t_i}}} \, . \end{equation}
The estimates~\eqref{equa1} and~\eqref{eq:apu5} imply
\begin{equation}
 C\cdot u(t_i) \ge \lf(\frac{t_i^\frac{1}{p}}{\exp\lf(\frac{\varepsilon (t_i)}{t_i}\r)^n}\r)^{\frac{p}{p+1-n}} \, .
 \end{equation}
Since $\alpha \ge n$ we have $\exp (1/t^n) \le \exp (1/t^\alpha)$ for $0<t\le 1$ and therefore
\[ \frac{C \cdot e} {\exp \left(  {}{t^{-n}} \right)} \ge  \lf(\frac{t_i^\frac{1}{p}}{\exp\lf(\frac{\varepsilon (t_i)}{t_i}\r)^n}\r)^{\frac{p}{p+1-n}} \, . \]
This means there are constants $C_1, C_2>0$ such that
\[\varepsilon (t_i) \ge C_1 \cdot t_i^n \log \left( C_2 \, t_i^\beta \exp (t_i^{-n})\right) \, , \quad \beta= \nicefrac{1}{p-n+1} \, . \]
Letting $i \to \infty$, the right hand hand converses to $C_1$ and $\varepsilon (t_i) \to 0$. This contradiction competes the proof.

\end{proof}
\subsection{The existence part of Theorem \ref{thm:cusp}}
\begin{thm}\label{inverse}
Let $u(t)=\frac{e}{\exp\lf(\frac{1}{t}\r)^\alpha}$ for some $0<\alpha<n$. Then there exists a homeomorphism $h \colon \BB\rightarrow\BB_u$ with finite conformal energy whose  inverse  $f=h^{-1} \colon \BB_u\to\BB$ is a Lipschitz regular.
\end{thm}
\begin{proof}
Fix $0<\alpha <n$ and the  corresponding cusp domain $\mathbb B_u$ with $u(\tau)=\frac{e}{\exp\lf(\tau^{-1}\r)^\alpha}$. As in the proof of Theorem~\ref{thm:62} we write
\[\R^n = \R \times \R^{n-1}= \{(t,x) \colon  t\in \R \, , \; x \in \R^{n-1}\}   \]
and replace $\mathbb B$  by a  bi-Lipschitz equivalent domain, $\X=\X_-\cup \X_+$, where
\[\X_- = \{ (t,x) \colon -1<t\le 0 \textnormal{ and } \abs{x}<1  \}\]
and
\[\X_+ = \{ (t,x) \colon 0< t<1 \textnormal{ and } t < \abs{x}<1  \} \, . \]
We replace the cusp domain $\mathbb B_u$ by the following bi-Lipschitz equivalent domain $\Y = \Y_- \cup \Y_+$, where
\[\Y_- = \{ (s,y) \colon -1<s\le 0 \textnormal{ and } \abs{y}<1 \}\]
and
\[ \Y_+= \{(s,y) \colon 0 < s <1 \textnormal{ and }  u(s)<\abs{y} <1 \}  \, . \]
We define $h \colon \X \onto \Y$ by
\[
h(t,x)= \begin{cases} (t,x) & \textnormal{in } \X_- \\
 \left(\frac{u^{-1} (\abs{x})}{\abs{x}}t, \, x\right)   \quad & \textnormal{in } \X_+ \, . \end{cases}
\]
Note that the inverse function $u^{-1} (\eta) = \log^{-\frac{1}{\alpha}} \left( \frac{e}{\eta}\right)$.
Then the inverse mapping $f=h^{-1} \colon \Y \onto \X$ takes the form
\[ 
f(s,y) =  \begin{cases} (s,y) & \textnormal{in } \Y_- \\
\left(\frac{\abs{y}}{u^{-1} (\abs{y})}s, \, y \right) \quad & \textnormal{in } \Y_+ \, .  \end{cases}
\]
Now, $f$ is a Lipschitz regular mapping. Furthermore, we have
\[\abs{Dh(t,x)} \le \frac{C}{\abs{x} \log^\frac{1}{\alpha} \left(\frac{e}{\abs{x}} \right)} \, . \]
Therefore,
\[\int_{\mathbb \X} \abs{Dh}^n < \infty   \]
as claimed.
\end{proof}

\end{document}